\documentclass[12pt]{degruyter-journal-a}      

\usepackage{graphicx, amsmath, tikz}
\usetikzlibrary{arrows}
\usetikzlibrary{decorations.markings}

\title{An Application of Abel's Method to the Inverse Radon Transform}

\lastnameone{Alimov}
\firstnameone{Shavkat}
\nameshortone{S.~Alimov}
\addressone{V.I. Romanovskiy Institute of Mathematics, Tashkent}
\countryone{Uzbekistan}
\emailone{sh\_alimov@yahoo.com}

\lastnametwo{David}
\firstnametwo{Joseph}
\nameshorttwo{J.~David}
\addresstwo{Florida State University, Tallahassee FL}
\countrytwo{United States}
\emailtwo{jkd14d@my.fsu.edu}

\lastnamethree{Nolte}
\firstnamethree{Alexander}
\nameshortthree{A.~Nolte}
\addressthree{Tufts University, Medford MA}
\countrythree{United States}
\emailthree{Alexander.Nolte@tufts.edu}

\lastnamefour{Sherman}
\firstnamefour{Julie}
\nameshortfour{J.~Sherman}
\addressfour{University of Minnesota - Twin Cities, Minneapolis MN}
\countryfour{United States}
\emailfour{sherm322@umn.edu}

\abstract{A method of approximating the inverse Radon transform on the plane by integrating against a smooth kernel is investigated. For piecewise smooth integrable functions, convergence theorems are proven and Gibbs phenomena are ruled out. Geometric properties of the kernel and their implications for computer implementation are discussed. Suggestions are made for applications and an example is presented.}



\researchsupported{This material is based upon work supported by the National Science Foundation under Grant No. NSF 1658672.}

\begin{document}

\section{Introduction}
\thispagestyle{empty}

 For any function $f\in L(\mathbb{R}^2)$ we consider its Radon transform $Rf$, which is defined for any straight line $L\subset \mathbb{R}^2$ as
 \[
 Rf(L) = Rf(t,\psi) = \int\limits_{-\infty}^\infty f(t\cos\psi +
 u\sin\psi, -t\sin\psi + u\cos\psi)\, du,
 \tag{1.1}\] where $0\leq\psi<\pi$ and $-\infty<t<+\infty$.
 
 The angle $\psi$ is defined by the unit vector $\emph{\textbf{n}}
 =(\cos\psi,\sin\psi)$ which is orthogonal to the straight line
 $L$, and $t$ is the real number such that the vector
 $t\emph{\textbf{n}}$ connects the origin with the line $L$.
 
 By Fubini's theorem, the Radon transformation (1.1) exists
 for almost all $\psi$ in $[0,\pi)$ and $t\in\mathbb{R}$.
 
 There is a well-developed theory surrounding the computation of
 the inverse Radon transform which has found many fruitful
 applications (see monographs \cite{Herman}, \cite{Kabanikhin},
 \cite{Lavrent'ev} \cite{Natterer}, \cite{Pickalov}). One such
 family of numerical methods finds approximations of a function by
 expanding it into a series in terms of functions whose inverse
 Radon transforms are known (\cite{Herman}, Ch. 6). Other
 well-studied numerical methods include the methods of linograms
 and $\rho$-filtered layergrams, algebraic reconstruction
 techniques, and simultaneous iterative reconstruction techniques
 (\cite{Herman}). Many of these numerical methods are based on
 Tikhonov's principle of regularization (\cite{Tikhonov GSY},
 \cite{Tikhonov LY}). A particularly popular family of approaches
 for applications in tomography are filtered backprojection
 algorithms, which use convolution against suitable functions to
 approximate steps in the process of finding the inverse Radon
 transform (\cite{Herman}, Ch. 8). Some other methods make use of
 convolution with singular functions (see \cite{Lavrent'ev}).
 
 In this paper, we introduce an approximation of the inverse Radon transform that uses integration against a smooth kernel. This approach is direct and does not suffer from Gibbs phenomena. The motivation for this approach comes from Abel's method of summation
 for integrals, and our method of proof leverages the interplay between the Fourier Transform and the Radon Transform (see \cite{Radon}).
 
 For any $\alpha>0$ we introduce the integral operator
 \[
 A_\alpha f(x)\ =\ \int\limits_{0}^\pi d\psi,
 \int\limits_{-\infty}^\infty \Phi_\alpha(x,t,\psi)\, Rf(t,\psi)
 dt\ . \tag{1.2}\label{Abel-means}
 \] Here,
 \[
 \Phi_\alpha(x,t,\psi)\ =\ \frac{1}{2\pi^2}\, \frac{\alpha^2 -
 	(x_1\cos\psi + x_2\sin\psi - t)^2}{[\alpha^2 + (x_1\cos\psi +
 	x_2\sin\psi - t)^2]^2} \ . \tag{1.3}
 \]
 
 \par Our central results will concern piecewise smooth functions.
 From the perspective of applications, real-valued piecewise smooth
 functions are the most important --- they are the sorts of
 functions that arise naturally with density distributions in
 tomography. We say that a function $f$ is a simple piecewise
 smooth function if $f$ is of the form\[f(x) = F(x) \chi_D(x)\]
 where $F$ is differentiable, $D \subset \mathbb{R}^2$ has a
 piecewise smooth boundary that does not intersect itself, and
 \[\chi_D(x) = \begin{cases}1 & x \in D \cup \partial D, \\ 0 & x
 \not\in D \cup \partial D. \end{cases} \] A function will be said
 to be piecewise smooth if it is a finite linear combination of
 simple piecewise smooth functions.
 
 \par For a piecewise smooth function $f$ and real $r > 0$,
 define \[S_r f(x) = \frac{1}{2\pi} \int\limits_{0}^{2\pi} f(x_1 +
 r\cos\theta, x_2 + r \sin\theta)\, d \theta,\] and \[Sf(x) =
 \lim_{r \rightarrow 0} S_rf(x).\] The operator $Sf$ can be seen as
 taking the local average of $f$. It is not difficult to prove the
 existence of $Sf(x)$ for any piecewise smooth $f$ and point $x \in
 \mathbb{R}^2.$ At points where $f$ is continuous, it is clear that
 $Sf(x) = f(x)$. Furthermore, if $f$ is continuous on an open set
 $\Omega$, then on any compact subset $K$ of $\Omega$, $S_rf
 \rightarrow f$ uniformly on $K$.
 
 As an example, if $Q$ is the unit square \[ Q = \{(x_1, x_2) \in
 \mathbb{R}^2 : \max(|x_1|, |x_2|) \leq 1\},\] then $\chi_Q$ is
 piecewise smooth and \[S\chi_Q((x_1, x_2)) = \begin{cases}
 1, & \max(|x_1|, |x_2|) < 1, \\
 \frac{1}{2}, & \max(|x_1|, |x_2|) = 1 \text{ and } \min(|x_1, |x_2|) < 1, \\
 \frac{1}{4}, & |x_1| = |x_2| = 1, \\
 0, & \max(|x_1|, |x_2|) > 1.
 \end{cases} \] For a visual interpretation, see figure \ref{square-example}.
 
 \begin{figure}[t]
 	\begin{center}
 		\begin{tikzpicture}
 		\draw[fill= gray!35] (0,0) rectangle (4,4);
 		
 		\draw[fill=black] (4,0) circle (.07cm);
 		\draw (4, 0) circle (.3cm);
 		
 		\draw[fill=black] (2.4, 3.2) circle (.07cm);
 		\draw (2.4, 3.2) circle (.3cm);
 		
 		\draw[fill=black] (0, 1.8) circle (.07cm);
 		\draw (0, 1.8) circle (.3cm);
 		
 		\draw[fill=black] (4.5, 1.4) circle (.07cm);
 		\draw (4.5, 1.4) circle (.3cm);
 		\end{tikzpicture}
 		\caption{There are four cases for local averages over circles on the unit square.}\label{square-example}
 	\end{center}
 \end{figure}
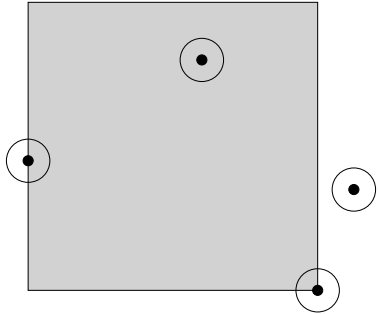
 
 In this paper, we will prove the following theorems.
 
 \medskip
 
 \begin{theorem} Let $f \in L(\mathbb{R}^2)$ be piecewise
 	smooth. Then for all $x \in \mathbb{R}^2$, \[\lim_{\alpha
 		\rightarrow 0} A_\alpha f(x) = Sf(x).\]
 \end{theorem}
 
 \begin{theorem}Let $f \in L(\mathbb{R}^2)$ be piecewise
 	smooth and $\alpha > 0$. If $|f(x)| < C$ for all $x \in
 	\mathbb{R}^2$, then for all $x \in \mathbb{R}^2$, \[|A_\alpha
 	f(x)| \leq C.\] Furthermore, if $f$ is real-valued and $A \leq
 	f(x) \leq B$ for all $x \in \mathbb{R}^2$, then for all $x \in
 	\mathbb{R}^2$, \[A \leq A_\alpha f(x) \leq B.\]
 \end{theorem}

 \begin{theorem} If $f\in L(\mathbb{R}^2)$ is piecewise
 	smooth and continuous on some open set $\Omega\subset
 	\mathbb{R}^2$ then  $A_\alpha f$ converges to $f$ uniformly on compact subsets of $\Omega$.
 \end{theorem}
  
 \begin{remark*} The kernel (1.3) is an improved version of the
 corresponding kernel which was proposed in \cite{Alimov}.
 \end{remark*}
 
 Our method to prove these theorems will be to find alternative representations of $A_\alpha f$ that connect it to the classical Fourier Abel means of $f$.
 
 \section{The Abel means of the inverse Radon transform }
 
 \
 Following the classical paper of J. Radon \cite{Radon}, we consider the
 Fourier transform
 \[
 \widehat{f}(\xi)\ =\ \int\limits_{\mathbb{R}^2} f(x)\,
 \exp(-ix\xi)\, dx.
 \] Let $\xi=s(\cos\psi,\sin\psi)$, where $s\in\mathbb{R}$ and
 $\psi\in[0,\pi)$. We change variables by
 \[
 y_1 = x_1\cos\psi + x_2\sin\psi, \quad y_2 = -\,x_1\sin\psi +
 x_2\cos\psi.
 \] The Jacobian determinant of this transformation is $1$, and its inverse transform is given by
 \[
 x_1 = y_1\cos\psi - y_2\sin\psi, \quad x_2 = y_1\sin\psi +
 y_2\cos\psi. \tag{2.1}
 \] Hence,
 \[
 \widehat{f}(\xi)\ =\ \int\limits_{-\infty}^\infty dy_1
 \int\limits_{-\infty}^\infty f(y_1\cos\psi - y_2\sin\psi,
 y_1\sin\psi + y_2\cos\psi)\exp(-ix\xi)\, dy_2 .
 \] Note that by (2.1)
 \[
 x\xi\ =\ (y_1\cos\psi - y_2\sin\psi, y_1\sin\psi +
 y_2\cos\psi)\cdot s(\cos\psi,\sin\psi)\ =\ y_1 s.
 \] Therefore,
 \[
 \widehat{f}(\xi)\ =\ \int\limits_{-\infty}^\infty \exp(-is y_1)\,
 dy_1 \int\limits_{-\infty}^\infty f(y_1\cos\psi - y_2\sin\psi,
 y_1\sin\psi + y_2\cos\psi)\, dy_2\ . \tag{2.2}
 \]
 
 Substituting the definition of the Radon transform (1.1) to (2.2), we have proven the following proposition.

 \begin{proposition}[J. Radon] If $f\in L(\mathbb{R}^2)$,
 \[
 \widehat{f}(\xi)\ =\ \int\limits_{-\infty}^\infty
 Rf(t,\psi)\, \exp(-is t) \, dt, \tag{2.3}
 \]
where $Rf$ is the Radon transform of $f$ and $\xi=s(\cos\psi,\sin\psi)$.
\end{proposition}

 We now consider the Fourier-Abel means of the inverse Fourier transform
 \[
 A_\alpha^F f(x)\ =\ \frac{1}{(2\pi)^2} \int\limits_{\mathbb{R}^2}
 \widehat{f}(\xi)\, \exp(ix\xi - \alpha|\xi|) \, d\xi. \tag{2.4}
 \] Our notation $A_\alpha^Ff$ is temporary: once the following proposition is proven, we will use $A_\alpha f$ interchangeably for $A_\alpha f$ and $A_\alpha ^Ff$.
 
 \begin{proposition} For any $f \in L(\mathbb{R}^2)$,  $A_\alpha f = A_\alpha^Ff.$
 \end{proposition}
 
 \begin{proof} We introduce new variables in (2.4):
 	\[
 	\xi=(s\cos\psi,s\sin\psi),\ 0\leq \psi<\pi,\ -\infty<s<+\infty.
 	\] Then,
 	\[
 	x\xi\ =\ s(x_1\cos\psi + x_2\sin\psi).
 	\] The Jacobian determinant corresponding to this change of variables is
 	\[
 	J(s,\psi)\ =\ \begin{vmatrix}\dfrac{\partial \xi_1}{\partial s}
 	\quad \dfrac{\partial \xi_1}{\partial \psi}\\
 	\\
 	\dfrac{\partial \xi_2}{\partial s}
 	\quad \dfrac{\partial \xi_2}{\partial \psi}\\
 	\end{vmatrix}\ =\ \begin{vmatrix}\cos\psi
 	\ (-s)\sin\psi\\
 	\\
 	\sin\psi
 	\quad\ s\cos\psi\\
 	\end{vmatrix}\ =\ s.
 	\] Hence, $|J(s,\psi)| = |s|$ and thus $d\xi = |s|ds\,d\psi$. Then by Proposition 2.1 and Fubini's theorem, denoting $(x_1 \cos \psi + x_2 \sin \psi)$ by $\beta = \beta(x,\psi)$ we see
 	\begin{align*}
 	A_\alpha^F f(x) &= \frac{1}{(2\pi)^2} \int\limits_{0}^\pi d\psi
 	\int\limits_{-\infty}^\infty \left(\int\limits_{-\infty}^\infty
 	\exp(-is t) Rf(t,\psi) dt\right) \exp(is\beta -
 	\alpha|s|)\, |s| ds \\
 	\tag{2.5} &= \frac{1}{(2\pi)^2} \int\limits_{0}^\pi d \psi \int\limits_{-\infty}^\infty \left( \int\limits_{-\infty}^\infty \exp(is(\beta - t) - \alpha |s|) \, |s| ds \right) Rf(t, \psi) dt.
 	\end{align*}
 	
 	Set
 	\[
 	\Phi_\alpha(x,t,\psi)\ =\ (2\pi)^{-2} \int\limits_{-\infty}^\infty
 	\exp(is(\beta - t) - \alpha|s|) \, |s| ds.
 	\tag{2.6}
 	\] By breaking up the integral in (2.6) and applying exponential identities, we see that
 	\[
 	\Phi_\alpha(x,t,\psi)\ =\ \frac{1}{2\pi^2} \int\limits_{0}^\infty
 	e^{- \alpha s} \cos [s(\beta - t)] \, s\, ds.
 	\tag{2.7}
 	\] Now, applying the known integral identity
 	\[
 	\int\limits_0^\infty e^{-\alpha s}\ \cos (\beta s)\, s\, ds\ =\
 	\frac{\alpha^2 - \beta^2}{(\alpha^2 + \beta^2)^2}, \quad \alpha>0\
 	, \quad \beta\in\mathbb{R},
 	\]
 	and substituting into (2.5), we obtain that
 	\[ A_\alpha^F f(x) = \int\limits_0^\pi d\psi \int\limits_{-\infty}^\infty \Phi_\alpha(x,t,\psi) Rf(t, \psi) dt = A_\alpha f(x).
 	\qedhere \] \end{proof}
 
\section{The kernel of Abel means}	
 \par In this section, we apply our previous results to prove the central findings of this paper. Set
 \[
 H_\alpha(x)\ =\ \frac{1}{2\pi}\ \frac{\alpha}{(\alpha^2 +
 	|x|^2)^{3/2}}, \quad x\in\mathbb{R}^2 . \tag{3.1}
 \] The Fourier transform of this function is
 \begin{align*}
 \widehat{H_\alpha}(\xi)\ =\ \frac{\alpha}{2\pi}
 \int\limits_{\mathbb{R}^2} \frac{\exp(-ix\xi)\, dx}{(\alpha^2 +
 	|x|^2)^{3/2}}\ =\ \frac{\alpha}{2 \pi} \int\limits_0^\infty
 \frac{r\,dr}{(\alpha^2 + r^2)^{3/2}}
 \int\limits_0^{2\pi} \exp(-ir|\xi|\cos\theta)\, d\theta. \tag{3.2}
 \end{align*}We will now make use of the facts that
 \[
 \frac{1}{2\pi} \int\limits_0^{2\pi} \exp(-ir|\xi|\cos\theta)\,
 d\theta\ =\ J_0(r|\xi|), \tag{3.3}
 \]
 where $J_0$ is the Bessel function of the first kind of order 0 and
 
 \[
 \alpha \int\limits_0^\infty \frac{J_0(r|\xi|)\,r\,dr}{(\alpha^2 +
 	r^2)^{3/2}}\ =\ \exp(-\alpha|\xi|). \tag{3.4}
 \]
 
 The equation (3.3) is known to follow from the definition of the
 Bessel functions (see \cite{Watson} formula 2.2(5)) and the
 equation (3.4) was proved by N. Sonine (see \cite{Watson},
 formulas 13.6(2) and 3.71(13)).

 \par Combining (3.2)-(3.4) shows that $\widehat{H_\alpha}(\xi)\ = \exp(-\alpha |\xi|)$. Therefore for any $f$ in $L(\mathbb{R}^2)$, \[A_\alpha f(x) = \frac{1}{(2\pi)^2}\int\limits_{\mathbb{R}^2} \widehat{f}(\xi) \widehat{H_\alpha}(\xi)\exp(ix\xi) d\xi = \frac{1}{(2\pi)^2}\int\limits_{\mathbb{R}^2}\widehat{f * H_\alpha}(\xi) \exp(i x \xi)\, d\xi.\] Since $f \in L(\mathbb{R}^2),$ $\widehat{f}$ is bounded, hence the map $\xi \mapsto  \widehat{f}(\xi)\exp(-\alpha |\xi|)$ is in $L(\mathbb{R}^2)$. Thus, the Fourier inversion theorem demonstrates that
 \[
 A_\alpha f(x)\ =\ (H_\alpha*f)(x)\ =\ \int\limits_{\mathbb{R}^2}
 H_\alpha(x-y)\, f(y)\, dy\
 \] almost everywhere. Since both $A_\alpha f$ and $H_\alpha * f$ are continuous, we see that $A_\alpha f(x) = (H_\alpha * f)(x)$ for all $x \in \mathbb{R}^2$. From this, we deduce the following proposition.
 
 \begin{proposition} For any $f \in L(\mathbb{R}^2)$,
 
 \[
 A_\alpha f(x)\ =\ \frac{\alpha}{2\pi} \int\limits_{\mathbb{R}^2}
 \frac{f(x+y)}{(\alpha^2 + |y|^2)^{3/2}}\, dy\ . \tag{3.5}\label{abel-convolution-rep}
 \]
 \end{proposition}
 
 \ We now establish Theorems $1.1-1.3$. The key remaining fact needed
 to prove these theorems is that the distributions $H_\alpha$ and
 $K_\alpha: \mathbb{R} \rightarrow \mathbb{R}$ given by \[ x
 \mapsto \begin{cases} \displaystyle\frac{\alpha x}{(\alpha^2 +
 	x^2)^{3/2}}, & x > 0 \\ 0, & x \leq 0 \end{cases}\] are
 $\delta$-shaped kernels. To show this, we must demonstrate
 \begin{enumerate}
 	\item $K_\alpha(x), \geq 0$ and for all $x \in \mathbb{R}$,
 	\item For any $\delta > 0$, $K_\alpha$ converges uniformly to $0$ outside of $[-\delta, \delta]$,
 	\item For all $\alpha > 0$, \[\int_{\mathbb{R}} K_\alpha(x) dx = 1. \tag{3.6}\label{helpful-integral}\]
 \end{enumerate} and that $H_\alpha$ satisfies analogous conditions.
 We begin with $K_\alpha$. $(1)$ is clear. To see $(2)$, Outside of $[-\delta, \delta]$, \[K_\alpha(x) \leq \frac{\alpha x}{(\alpha^2 + x^2)^{3/2}} \leq \alpha \frac{1}{\delta^2}, \] which converges uniformly to $0$ outside of $[-\delta, \delta]$. For $(3)$, use the substitution $u = \alpha^2 + r^2$ to see \[\int\limits_{\mathbb{R}}K_\alpha dx = \alpha \int\limits_0^\infty \frac{x \, dx}{(\alpha^2 + x^2)^{3/2}} = -\frac{\alpha}{\sqrt{\alpha^2 + r^2}} \bigg|_{r=0}^{r=\infty} = 1.\] Thus, $K_\alpha$ is $\delta$-shaped.
 
 \par For $H_\alpha$, $(1)$ is also clear. To see $(2)$, let $\delta > 0$ be given. Then, whenever $|x| > \delta$,
 \[
 H_\alpha(x)\ \leq\ \frac{1}{2\pi}\ \frac{\alpha}{(\alpha^2 +
 	\delta^2)^{3/2}}\ \leq\ \frac{1}{2\pi}\ \frac{\alpha}{\delta^3}.
 \] Thus, $H_\alpha(x)\to 0$ uniformly as $\alpha\to0$ outside of $[-\delta, \delta]$. For $(3)$, writing the integral in polar coordinates and comparing to (\ref{helpful-integral}) shows
 \[
 \int\limits_{\mathbb{R}^2} H_\alpha(x)\, dx\ =\ \alpha
 \int\limits_0^\infty \frac{r\, dr}{(\alpha^2 + r^2)^{3/2}}\ =1,
 \] completing the proof that $H_\alpha$ is $\delta$-shaped.
 
 If $f$ is piecewise smooth and $x \in \mathbb{R}^2$ is fixed, writing the representation in (\ref{abel-convolution-rep}) in polar coordinates, we see that \begin{align*}A_\alpha f(x) &= \frac{\alpha}{2\pi} \int\limits_0^\infty \frac{r}{(\alpha^2 + r^2)^{3/2}} \left(\int\limits_0^{2 \pi} f(x + (r \cos \theta, r \sin \theta)) d \theta \right) dr \\ &= \alpha \int\limits_0^\infty \frac{r S_rf(x)}{(\alpha^2 + r^2)^{3/2}} dr. \tag{3.7}\label{surprise-Sf}\end{align*} %
 %
 %
 %
 %
 
 Now, adding and subtracting $Sf(x)$ from (\ref{surprise-Sf}) we obtain that \begin{align*}A_\alpha f(x) &= S f(x) + \alpha \int\limits_{0}^\infty \frac{r[S_rf(x) - Sf(x)]}{(\alpha^2 + r^2)^{3/2}}dr\\
  &= Sf(x) + \int\limits_{\mathbb{R}} K_\alpha(r)\left[S_rf(x) - Sf(x) \right]dr.\tag{3.8}\label{final-step-major-result}\end{align*}
 
 If we let $S_x : \mathbb{R} \rightarrow \mathbb{C}$ be the map $r \rightarrow S_{|r|}f(x) - Sf(x)$, (\ref{final-step-major-result}) precisely says \[A_\alpha f(x) = Sf(x) + (K_\alpha * S_x)(0).\] Since $f$ is piecewise smooth, as $r \rightarrow 0$, $S_{|r|}f(x) \rightarrow Sf(x)$. Hence $S_x$ is continuous at $0$ with $S_x(0) = 0$. Since $K_\alpha$ is $\delta$-shaped, (\ref{final-step-major-result}) shows \[\lim_{\alpha \rightarrow 0} A_\alpha f(x) = \lim_{\alpha \rightarrow 0} Sf(x) + (K_\alpha * S_x)(0) = Sf(x) - S_x(0) = Sf(x),\] proving Theorem $1.1$.

 \par To show Theorem $1.2$, suppose that $f \in L(\mathbb{R}^2)$ is piecewise smooth and $|f(x)| \leq C$ for all $x \in \mathbb{R}^2.$ Then, \[|A_\alpha f(x)| \leq \int\limits_{\mathbb{R}^2} |f(x - y)| H_\alpha(y) dy \leq C \int\limits_{\mathbb{R}^2} H_\alpha(y)dy = C,\] where the last two steps are due to the fact that $H_\alpha$ is $\delta$-shaped. Identical arguments show the remainder of theorem $1.2$'s assertions.

 \par To prove Theorem $1.3$, let $f$ be peicewise smooth, continuous on an open set $\Omega$ and $K \subset \Omega$ be compact. Then the theorem can be seen as a consequence of (\ref{surprise-Sf}) and the fact that $S_rf \rightarrow f$ uniformly on $K$.
 
\section{Discussion of the Kernel}

 \par The kernel $\Phi_\alpha$ that appears in the Abel means (\ref{Abel-means}) has some properties that demand attention in computational applications of this method.
 We discuss these and present suggestions. For a concrete example
 of these patterns, see Section 5. \par If we regard $x$ and $\psi$
 as fixed, as is the case within the inner integral of (1.1), and
 set $\beta = x_1 \cos(\psi) + x_2 \sin(\psi),$ we obtain that
 \[\Phi_\alpha(x, t, \psi) = \frac{1}{2\pi^2}\frac{\alpha^2 -
 	(\beta - t)^2}{(\alpha^2 + (\beta - t)^2)^2}.\] Figure
 (\ref{phi-alpha-plot}) shows the shape of this distribution.
 Differentiating with respect to $t$ and substituting critical
 values yields the following features of $\Phi_\alpha$ with fixed
 $x$, $\psi$:
 
 \begin{enumerate}
 	\item The unique maximum of $\Phi_\alpha$ occurs at $t = \beta$, with $2 \pi^2 \Phi_\alpha(x,\beta,\psi) = (1/\alpha)^2$;
 	\item The two minima of $\Phi_\alpha$ occur at $\beta + \alpha\sqrt 3$ and $\beta- \alpha\sqrt 3 $, each with value $(-1/8)\Phi_\alpha(x, \beta, \psi)$;
 	\item $\Phi_\alpha$ tends to $0$ monotonically outside of $[\beta - \alpha{\sqrt 3}, \beta + \alpha\sqrt 3]$ with speed on the order of $(1/t)^2.$
 \end{enumerate}
 As $\alpha$ tends to $0$, this results in a small area centered around $\beta$ in which $\Phi_\alpha$ is extremely large, and outside of which $\Phi_\alpha$ is small.

 \begin{figure}[t]
 	\includegraphics[scale = 0.6]{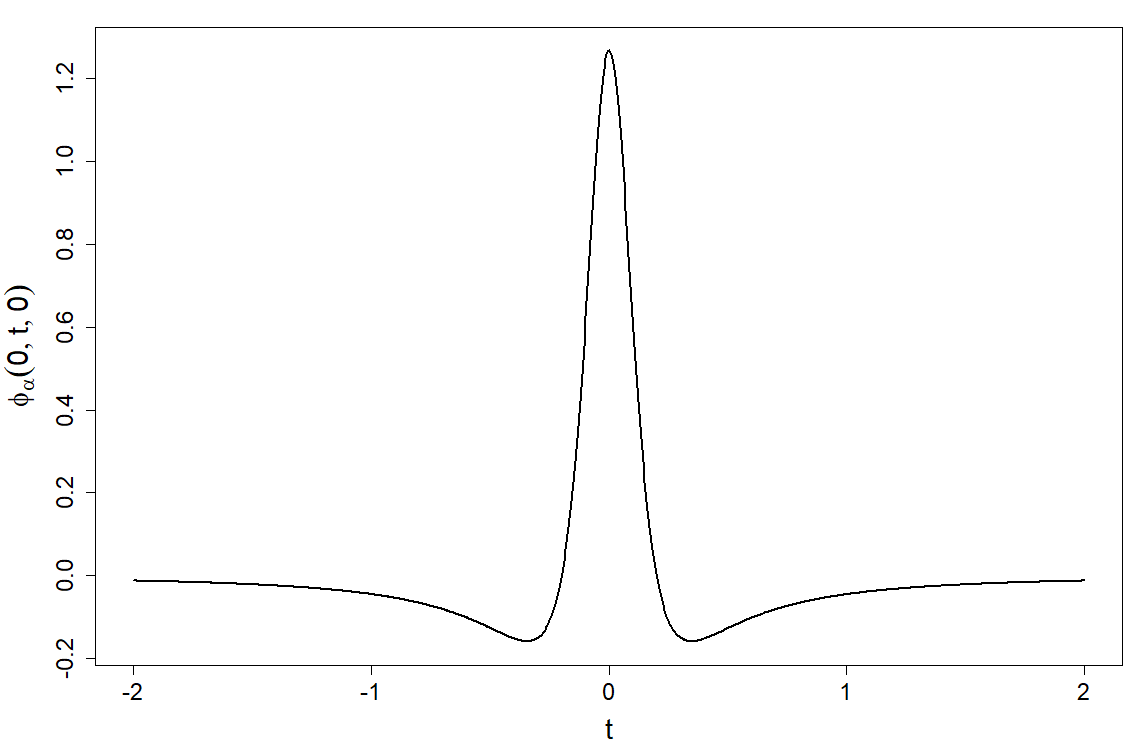}
 	\caption{The kernel $\Phi_\alpha$ with $x = 0$, $\psi = 0$ and $\alpha = 0.2$. As $\alpha$ grows smaller, the peak rapidly becomes larger and narrower.}\label{phi-alpha-plot}
 \end{figure}

 \par This presents problems to computer integration procedures if no modifications are made. This is because integration procedures approximate the values of functions on intervals with samples. When the peak is small enough, it becomes likely for the procedure to never sample the peak, and if a sample does happen to hit the peak, the procedure will strongly overestimate the peak's contribution. The effect of this is a poor reconstruction of the original function with high variance.
 
 \par Our knowledge of the behavior of $\Phi_\alpha$ gives a method to mitigate this problem. Since $\Phi_\alpha$ is large close to $t = \beta$ and small far from it, calculating the contribution of the area close to this critical value with a small step size separately from the rest of $\Phi_\alpha$ diminishes this error. The simplest way to accomplish this is to have a computer evaluate $A_\alpha$ as three integrals of the form
 \begin{align*}
 \int\limits_{0}^\pi \int\limits_{-\infty}^{\beta - \epsilon} \Phi_\alpha(x,t,\psi)Rf(t,\psi) dt d\psi &+\int\limits_{0}^\pi \int\limits_{\beta - \epsilon}^{\beta + \epsilon} \Phi_\alpha(x,t,\psi)Rf(t,\psi) dt d\psi \\
 \tag{4.1}&+ \int\limits_0^\pi \int\limits_{\beta + \epsilon}^\infty \Phi_\alpha(x,t,\psi)Rf(t,\psi) dt d\psi,
 \end{align*} where $\beta = \beta(x, \psi)$ is as above and the choice of an optimal $\epsilon > 0$ will be dependent on the desired $\alpha$ value and context of the application. The example in Section $5$ demonstrates how substantial a change this minor modification with $\epsilon = 2 \alpha$ makes to the accuracy with which a computer-evaluated $A_\alpha$ can reconstruct a function.
 
 \par With this change, computer calculations of $A_\alpha$ can remain accurate for much smaller $\alpha$ without substantially increasing the number of samples needed. However, the issue will still appear with small enough $\alpha$ since the area on which $\Phi_\alpha$ is large diminishes at a rate on the order of $\alpha$ while the peak grows at a rate on the order of $(1/\alpha)^2$.
 
 
 \section{Example}
 
 For positive $\rho_1$ and $\rho_2$ in  $(0, 1)$, consider the following discs:
 
 \[
 f(x,y)\ =\ \begin{cases} 1, & x^2 + y^2 \leq \rho_1^2,\\
 0, & x^2 + y^2 > \rho_1^2;
 \end{cases} \tag{5.1}
 \]

 \[
 g(x,y)\ =\ \begin{cases} 1, & (x-1)^2 + y^2 \leq \rho_2^2,\\
 0, & (x-1)^2 + y^2 > \rho_2^2.
 \end{cases} \tag{5.2}
 \]
 
 An elementary trigonometric argument demonstrates that these functions have the following Radon transforms:
 
 \[
 Rf(t,\psi)\ =\ \begin{cases} 2\,\,\sqrt[]{\rho_1^2-t^2}, & |t|\leq \rho_1,\\
 0, & |t| > \rho_1.
 \end{cases} \tag{5.3}
 \]
 
 
 \[
 Rg(t,\psi)\ =\ \begin{cases} 2\,\sqrt{\rho_2^2
 	-[t-\cos(\psi)]^2}, & \rho_2 \leq|t-\cos(\psi)|,\\
 0, &  \rho_2 >|t-\cos(\psi)|.
 \end{cases} \tag{5.4}
 \]

 Moreover, $R(f+g)(t,\psi) = Rf(t,\psi)+Rg(t,\psi)$.
 
 \par Taking $\rho_1 = 2$ and $\rho_2 = \frac{1}{2}$, we use numerical integration methods in an attempt to recover $f+g$ using (1.2), (5.3), and (5.4) by taking sufficiently small values of $\alpha$. Computations were completed using the programming language R
 with the function integral2 from the \textit{pracma} package.
 The results appear in Figure \ref{Radons} for varied values of
 $\alpha$ and methods of integration.
 
 \begin{figure}[t]
 	\includegraphics[scale = 0.62]{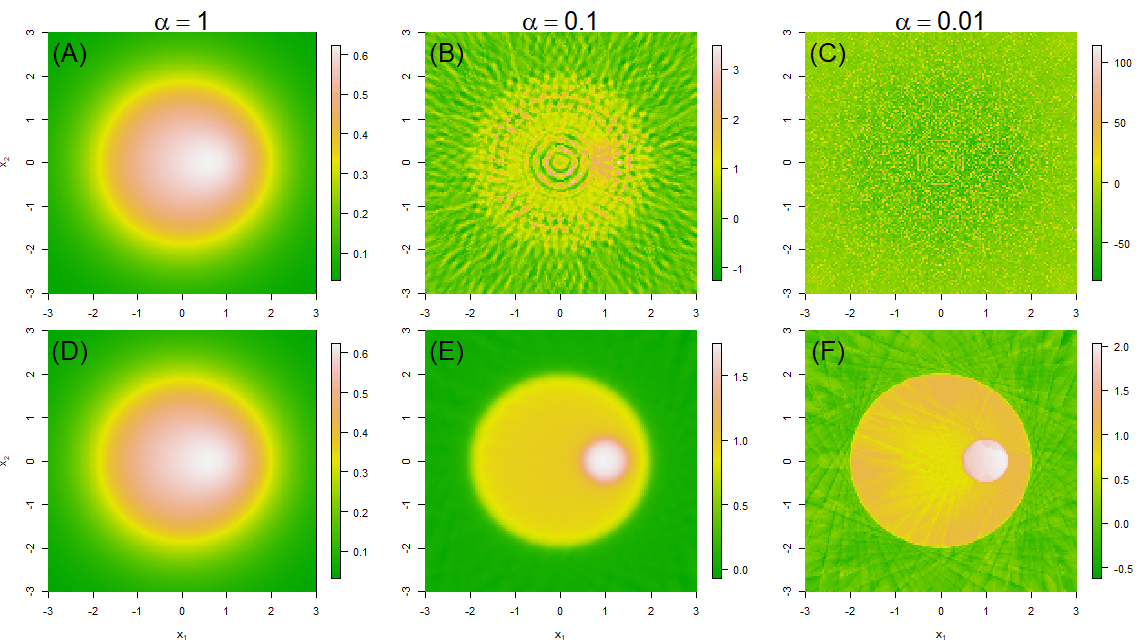}
 	\caption{Inverse Radon Transformation of overlayed discs (5.1) and (5.2). (A) - (C) computed by a single integral, (D) - (F) computed by three integrals as described in Section 4. (A) and (D) are computed with $\alpha = 1$, (B) and (E) are computed with $\alpha = 0.1$, and  (C) and (F) are computed with $\alpha = 0.01$. }\label{Radons}
 \end{figure}

 \par As discussed in Section 4, Figure \ref{Radons} (B) and (C) give evidence that for small values of $\alpha$ and an unmodified numerical integration technique, the resulting inverse transformation yields inaccurate results. However, these results are drastically improved by using the method in (4.1), as evidenced by comparison with Figure \ref{Radons} (E) and (F). Although not depicted here, it is important to note that for much smaller $\alpha$, the result when using the method in (4.1) eventually becomes inaccurate with high variance as well.

 \begin{center}
 	\textbf{6. Conclusion}
 \end{center}

 \par In this paper we have formulated a direct method for approximation of the inversion of the Radon
 transformation: we reconstruct the function from its Radon
 transform without computing the function's Fourier
 transform or solving an auxiliary system of equations. This avoids computing oscillatory
 integrals with high frequency or convolutions with functions with singularities.
 
 We can also note that due to Theorem $1.2$, Gibbs phenomena do not
 occur with the Abel means of piecewise smooth functions. It is
 important to note that in applications, the domain of integration
 is bounded and all functions are also bounded. Since the kernel
 (1.3) is bounded for every $\alpha > 0$, the integral (1.2) is not
 improper and hence can be evaluated with standard methods of
 integration.
 
 Following investigation of related problems may include finding
 fast and efficient methods for minimizing computational error, as
 well as computing the optimal value of $\alpha$ for such methods.

 \bibliographystyle{amsplain}

\end{document}